\documentclass[11pt]{amsart}
\usepackage{amssymb, latexsym, amsmath, amsfonts, tikz, tikz-cd}
\usepackage{hyperref, enumitem, fancyhdr}

\newtheorem{thm}{Theorem}[section]

\newtheorem{lem}[thm]{Lemma}
\newtheorem{prop}[thm]{Proposition}
\theoremstyle{definition}
\newtheorem{defn}[thm]{Definition}
\theoremstyle{remark}

\numberwithin{equation}{section}
\theoremstyle{remark}
\newtheorem{exam}[thm]{Example}

\newcommand{\mbb}{\mathbb}
\newcommand{\ra}{\rightarrow}

\newcommand{\pa}{\partial}
\newcommand{\ov}{\overline}
\newcommand{\sm}{\setminus}
\newcommand{\ep}{\epsilon}
\newcommand{\no}{\noindent}
\newcommand{\al}{\alpha}

\newcommand{\cal}{\mathcal}
\newcommand{\ti}{\tilde}
\newcommand{\la}{\lambda}

\newcommand{\be}{\beta}

\begin{document}
\title{Few remarks on the Poincar\'{e} metric on a singular holomorphic foliation}
\thanks{The author is supported by the Labex CEMPI (ANR-11-LABX-0007-01)}
\author{Sahil Gehlawat}

\address{Universit\'{e} de Lille, Laboratoire de Math\'{e}matiques Paul Painlev\'{e}, CNRS U.M.R. 8524, 59655
Villeneuve d’Ascq Cedex, France.}
\email{sahil.gehlawat@univ-lille.fr}

\begin{abstract} 
 Let $\mathcal F$ be a Riemann surface foliation on $M \sm E$, where $M$ is a complex manifold and $E \subset M$ is a closed set. Assume that $\cal{F}$ is hyperbolic, i.e., all leaves of the foliation $\cal{F}$ are hyperbolic Riemann surface. Fix a hermitian metric $g$ on $M$. We will consider the Verjovsky's modulus of uniformization map $\eta$, which measures the largest possible derivative in the class of holomorphic maps from the unit disk into the leaves of $\cal{F}$. Various results are known to ensure the continuity of the map $\eta$ along the transverse directions, with suitable conditions on $M$, $\cal{F}$ and $E$. For a domain $U \subset M$, let $\cal{F}_{U}$ be the holomorphic foliation given by the restriction of $\cal{F}$ to the domain $U$, i.e., $\cal{F}\vert_{U}$. We will consider the modulus of uniformization map $\eta_{U}$ corresponding to the foliation $\cal{F}_{U}$, and study its variation when the corresponding domain $U$ varies in the Caratheodory kernel sense, motivated by the work of Lins Neto--Martins.
\end{abstract}  

\maketitle 

\section{Introduction}

Let $M$ be a complex manifold of dimension $N$, and $E \subset M$ a closed subset of codimension atleast 2. Recall that a singular holomorphic foliation by Riemann surfaces $\cal{F}$ with singular set $E$ is defined by an atlas of flow-boxes $(U_{\al}, \phi_{\al})$ where $\phi_{\alpha} : U_{\alpha} \ra \mbb D \times \mbb D^{N-1}$ ($\mbb D \subset \mbb C$ is the open unit disc) and the transition maps 
$\phi_{\al\be}$ are holomorphic in $(x, y) \in \mbb D \times \mbb D^{N-1}$ and are of the form
\[
\phi_{\al \be} = \phi_{\al \be}(x, y) = \phi_{\al} \circ \phi^{-1}_{\be}(x, y) =  (P(x, y), Q(y)).
\]
The leaves of $\mathcal F$ are Riemann surfaces, which are locally given by $\phi^{-1}_{\al}(\{y = \mbox{constant} \})$. A leaf passing through a given $p \in M \sm E$ will be denoted by $L_p$. A hyperbolic Riemann surface foliation $(M, E, \mathcal F)$ is a singular holomorphic foliation in which every leaf $L_{p}$ of $\cal{F}$ is a hyperbolic Riemann surface.
 Fix a hermitian metric $g$ on $M$ and denote the length of a tangent vector $v$  by $\vert v \vert_{g}$. Consider the space $\mathcal{O}(\mbb{D}, M\sm E)$, which is the set of all holomorphic maps from the unit disc $\mbb D \subset \mbb C$ into the complex manifold $M \sm E$. Now consider the subset $\mathcal{O}(\mbb{D}, \cal{F}) \subset \mathcal{O}(\mbb{D}, M\sm E)$, the set of holomorphic maps $f \in \mathcal{O}(\mbb{D}, M\sm E)$ with values in a leaf of the foliation $\cal{F}$, that is $f(\mbb{D}) \subset L$, for some leaf $L$ of $\cal{F}$. Consider the map $\eta : M \sm E \ra (0, \infty)$ given by
\[
\eta(z) := \sup \left \{ \vert f'(0) \vert_g : f \in  \mathcal O(\mbb D, \mathcal F), f(0) = z \right\}.
\]
One can check that $\eta > 0$, and that it achieves its extremal value with the map $\pi_z$, where $\pi_z : \mbb D  \ra L_z$  is the universal covering map chosen such that $\pi_z(0) = z$ . Thus, $\eta(z) = \vert \pi'_z(0) \vert_g$. For a leaf $L \subset \cal{F}$ and a vector $v$ tangent to $L$ at $z$, let $\la_{ L}$ denote the Poincar\'{e} metric on $L$ and $\vert v \vert_{L}$ denote the Poincar\'{e} length of $v$. The extremal property of the Kobayashi metric shows that $\vert v \vert_g = \eta(z) \vert v \vert_{L}$. Symbolically, this will be written as $g = \eta \cdot \la_{L}$, and it can be understood that the restriction of $g/\eta$ to a leaf $L$ induces the corresponding Poincar\'{e} metric $\la_{L}$ on it.

\medskip

Since the Poincar\'{e} metric $\la_{L}$ for a leaf $L$ is a positive real-analytic function on $L$ and $g$ is smooth on $M$, therefore $\eta$ is also smooth along leaves. However, the regularity of $\eta$ along transverse directions is not a priori clear. In fact, the regularity of $\eta$ is related to the variation of $\la_{L}$ along such directions. This question of the regularity of $\eta$ corresponding to a hyperbolic singular foliation $\cal{F}$ has been studied under various hypotheses on $M, E$ and 
$\mathcal F$. One can have a look at Verjovsky \cite{V}, Lins Neto \cite{N1, N2}, Candel \cite{Ca} in which the continuity of $\eta$ on $M \sm E$ is established. Recently, Dinh--Nguy\^{e}n--Sibony \cite{DNS1, DNS2} gave a stronger estimate on the modulus of continuity 
of $\eta$ under suitable hypotheses on $M, E$ and $\mathcal F$. Related work on this theme can be found in \cite{CLS}, \cite{CG}, \cite{GV2}, \cite{FB} and \cite{FS}. To see few ramifications of studying the transverse regularity of the leafwise Poincar\'{e} metric, one can look at \cite{Ng1}, \cite{Ng2}, where Nguyen highlights its utility in studying the ergodic properties of Riemann surface laminations.

\medskip

In \cite{NM}, the variation of $\eta$ with respect to its domain of definition was studied. To describe it, let $U \subset M$ be an open subset, denote the restriction of the foliation $\mathcal F$ to $U$, that is $\cal{F}\vert_{U}$ by $\mathcal {F}_U$. Let $\eta_{U} : U \ra (0, \infty)$ be the positive function corresponding to the foliation $\mathcal F_U$. Note that $\eta_{U}$ is defined using holomorphic maps from $\mbb D$ into $L_{p, U}$ as $p$ varies in $U$. The following facts were established in \cite{NM} by Lins Neto--Martins. First, that $\eta_{U}$ is a monotone function of $U$, i.e., for domains $U_{1} \subset U_{2}$, $\eta_{U_{1}}(p) \le \eta_{U_{2}}(p)$ for all $p \in U_{1}$. Also, if $\{W_n\}$ is an increasing exhaustion of $W$ and the singular set of foliation $\cal{F}$ is discrete, then $\eta_{W_n} \ra \eta_{W}$ pointwise on $W \sm E$. The convergence is uniform on compact subsets of $W \sm E$ if the functions $\eta_{W_n}$ and 
$\eta_{W}$ are continuous. Similar observations about the pointwise convergence of $\eta_{W_n}$ were made by the author (with co-author Kaushal Verma) in \cite{GV1}, where the singular set $E$ is assumed to be discrete, $W$ be a bounded taut domain and $W_n$ converges to $W$ in a Hausdorff sense. The uniform convergence of $\eta_{W_n}$ to $\eta_W$ on the compacts of $W \sm E$ was also established in \cite{GV1}, with the additional hypotheses that $\eta_W$ is continuous on $W \sm E$ and that there exist another bounded taut domain $V \subset M$ such that $\ov{W} \subset V$. In \cite{GV2}, these observations were proved to be true for the case when $E$ is not discrete. It was also proved that if the foliation $\cal{F}$ is transversal type (see \cite{GV2} for the definition), then the uniform convergence will be on the compacts of $W$ instead of $W \sm E$.

\medskip
 The aim of this note is to study the variation of the function $\eta_{W_n}$ of the foliation $\cal{F}_{W_n}$ when the domains $W_n$ converges to the domain $W$ in the kernel sense (see below for definitions). We will observe that outside a \textit{defective} set, the convergence results still holds true. We will also see that these results are a generalization to the similar observations made in \cite{GV1,GV2} for the case when domain $U$ varies in Hausdorff sense, and in \cite{NM} where $U$ varies in monotonic increasing fashion. To keep things simple, unless otherwise mentioned we will fix $M = U \subset \mbb{C}^N$, where $U$ is a connected open subset with the metric $g$ to be the Euclidean metric throughout the article. Consider $W_n , W \subset U \subset \mbb{C}^N$ be connected open subsets with $0 \in W_n, W$ for all $n$. Let $\cal{F}$ a hyperbolic singular holomorphic foliation on $U$ with singular set $E \subset U$. Throughout the article, we will denote the foliation $\cal{F}_{W_n}$ by $\cal{F}_n$, the function $\eta_{W_n}$ by $\eta_{n}$, and the leaf $L$ of $\cal{F}$ on $U$ passing through $p \in U$ by $L_{p}$.

\begin{defn}
Let $W_n \subset U$ be a sequence of connected open subsets of $U$ such that $0 \in W_n$ for all $n$. Let $\tilde{V}_n := \cap_{k \ge n} W_k$ and $V_n$ be the connected component of the interior of $\ti{V}_n$ containing $0 \in U$, that is $V_n = (\text{int}(\ti{V}_n))_{0}$. The \textit{kernel} of the sequence $\{W_n\}_{n \in \mbb{N}}$, denoted by $W$ is defined to be the $\cup_{n \ge 1} V_n$ if it is non-empty, otherwise we say it is $\{0\}$.

\no We will say that $W_n$ converges to $W$, if every subsequence of $\{W_n\}_{n \in \mbb{N}}$ has the same kernel $W$.
\end{defn}

\no Now we can define the kernel convergence of a sequence of domains.

\medskip

\begin{defn}
Let $W_n \subset U \subset \mbb{C}^N$ be a sequence of connected open subsets of $U$ such that $0 \in W_n$ for all $n$, and $W \subset U$ be another connected open set with $0 \in W$. We say that $W_n$ converges to $W$ in the kernel sense if $W$ is the kernel of the sequence $\{W_n\}_{n \in \mbb{N}}$ and $W_n$ converges to its kernel.
\end{defn}

\medskip

\no\textbf{Note:} We can check that if $W_n \to W$ in the kernel sense, then for any compact set $K \subset W$, there exist $N_K \in \mbb{N}$ such that $K \subset W_n$ for all $n \ge N_K$.

\medskip
The following example shows that even the pointwise convergence of $\eta_n$ to $\eta_W$ is not guaranteed on the whole set $W \sm E$, when $W_n \to W$ in the kernel sense.


\begin{exam}
Let $U = P(0,(5,5)) \subset \mbb{C}^2$ be the polydisc of polyradius $(r_1, r_2) = (5, 5)$ and centre $0 \in \mbb{C}^2$. Take $W_n = P(0,(1,1)) \cup P(0, (2, \frac{1}{n}))$ and $W = P(0,(1,1))$. Now consider the vector field $X(z) = z$ on $U$ and let $\cal{F}, \cal{F}_n, \cal{F}_{W}$ be the foliation induced by $X$ on the set $U, W_n, W$ respectively. Here, the singular set of the foliation is the singleton set $E = \{0\}$. Since $U$ is bounded, the foliations $\cal{F}, \cal{F}_n, \cal{F}_{W}$ are all hyperbolic. 

\no Observe that $P(0, (2, \frac{1}{n})) \searrow D(0,2) \times \{0\}$ as $n \to +\infty$, that is 
\[
\cap_{k \ge n} P(0, (2, \frac{1}{k})) = D(0,2) \times \{0\}.
\]
This tells us that the set $\tilde{V}_n = \cap_{k \ge n} W_n$ satisfies $\text{int}(\ti{V}_n) = P(0, (1,1)) = W$. Therefore, the kernel of the sequence $W_n$ is $W$. One can easily check that this is true for all subsequences of $\{W_n\}_{n \in \mbb{N}}$. Thus, $W_n \to W$ in the kernel sense. 

\medskip
\no Now take $p = (\frac{1}{2},0) \in U$, and observe that $p \in W_n , W$ for all $n \ge 1$. Denote the leaf of foliation $\cal{F}$ (or $\cal{F}_n, \cal{F}_{W}$) passing through point $p$ by $L_p$ (or $L_{p,n}, L_{p,W}$ respectively). One can easily check that $L_{p,n} = D(0, 2)\sm\{0\} \times \{0\} \cong D(0,2) \sm \{0\}$, and $L_{p,W} = D(0,1)\sm\{0\} \times \{0\} \cong \mbb{D}^{*}$. Since the Poincar\'e metric of the punctured disc $\mbb{D}^{*}$ is given by 
\[
ds^2 = \frac{4}{\vert q\vert^2 (\log{\vert q\vert^2)^2}} \vert dq\vert^2,
\]
and on $D(0,2)\sm \{0\}$ by
\[
ds^2 = \frac{4}{\vert q\vert^2 (\log{\frac{\vert q\vert^2}{4})^2}} \vert dq\vert^2.
\]
Therefore
\[
\eta_n^{2}(p) = \frac{(\frac{1}{2})^2 (\log{16})^2}{4}, \; \eta_{W}^{2}(p) = \frac{(\frac{1}{2})^2 (\log{4})^2}{4}.
\]
Thus, $\eta_n(p)$ doesn't converge to $\eta_{W}(p)$, although $W_n \to W$ in the kernel sense.

\end{exam}

\medskip
But all is not lost though. We will see that the pointwise convergence will be true away from a subset which we will call as \textit{defective set}. 

\begin{defn}
Consider $(W_n,W,U,\cal{F})$ be as in the above discussion. Now for a subsequence $\sigma = \{n_k\}_{k \in \mbb{N}}$ of the sequence $\{n\}_{n \in \mbb{N}}$, consider
\[
F_{\sigma} := \{p \in U \sm \ov{W} : \exists \{p_{n_k}\}_{k \in \mbb{N}} \subset U \ \text{such that} \ p_{n_k} \in W_{n_k} \sm \overline{W} \ \forall \ k, \ \text{and} \ p_{n_k} \to p\}.
\]
Let $F := \cup_{\sigma} F_{\sigma}$, where $\sigma$ varies over all possible subsequence. The set $S := \cup_{q \in F} (L_q \cap W) \subset W$ will be called as the \textit{defective set} corresponding to the foliation $\cal{F}$.  
\end{defn}

\no Note that $S \subset W$ is a $\cal{F}_{W}-$invariant set since it is union of leaves in $W$.

\medskip

\begin{thm}\label{T:PointConv}
Let $(W_n,W,U,\cal{F})$ be as above with the additional hypothesis that $W$ is a bounded taut domain. If $W_n \to W$ in the kernel sense, then $\eta_n(p) \to \eta_W(p)$ for all $p \in W \sm (S \cup E)$.
\end{thm}

\medskip
The following proposition will be used in proving the above theorem.

\begin{prop}
Let $(W_n,W,U,\cal{F})$ be as above. If $W$ is the kernel of the sequence $\{W_n\}_{n \in \mbb{N}}$, then 
\[
\liminf_{n \to +\infty}{\eta_n(p)} \ge \eta_{W}(p),
\]
for all $p \in W \sm E$.
\end{prop}

\begin{proof}
Let $p \in W \sm E$, and $\alpha : \mbb{D} \to L_{p,W}$ be the uniformization of the leaf $L_{p,W}$ of the foliation $\cal{F}_{W}$. For $\ep > 0$, consider the map $\psi : \mbb{D} \to \mbb{D}$ given by $\psi(z) = (1-\ep)z$. Now consider the map $\alpha \circ \psi : \mbb{D} \to L_{p,W}$ and observe that $\alpha \circ \psi(\mbb{D})$ is relatively compact in $W$. From the definition of kernel, we know $W = \cup_{n \ge 1} V_n$, where $V_n$ satisfies $V_k \subset V_{k+1}$ for all $k$. This tells us that there exist $N \in \mbb{N}$ such that $\alpha \circ \psi(\mbb{D}) \subset V_{N}$, which in turn gives us that $\alpha \circ \psi(\mbb{D}) \subset W_{k}$ for all $k \ge N$. Thus $\alpha \circ \psi(\mbb{D}) \subset L_{p} \cap W_{k}$ and $p \in \alpha \circ \psi(\mbb{D})$. Since $L_{p,k}$ is the connected component of $L_{p} \cap W_k$ containing $p$, therefore $\alpha \circ \psi(\mbb{D}) \subset L_{p,k}$ for all $k \ge N$. By definition of the map $\eta_{k}$, we get for $k \ge N$
\[
\eta_k(p) \ge \vert (\alpha \circ \psi)'(0)\vert = (1 - \ep) \eta_{W}(p).
\]
This gives us that $\liminf_{n \to +\infty}{\eta_n(p)} \ge \eta_{W}(p)$ for all $p \in W \sm E$.
\end{proof}

\no Now we shift our focus on the uniform convergence of $\eta_n$. First note that we can extend the map $\eta$ on the singular set $E$ by defining $\eta(p) = 0$ for all $p \in E$. Before we state our next result, lets quickly recall what we mean by a transversal type foliation. For a singular point $p \in E$ of a holomorphic foliation $\cal{F}$ on a complex manifold $M$, we consider the cone $C_{p}(\cal{F}) \subset T_{p}{M}$ which is the collection of all limit points of the vectors coming from the tangent line bundle of the foliation $T\cal{F} = \cup_{q \in M \sm E} T_{q}{\cal{F}}$ on $M \sm E$. Also, since $E \subset M$ is an analytic subset, we will denote the Whitney's $C_4$-tangent cone (see \cite{Ch}) at $p \in E$ by $C_{p}E$. We say that $\cal{F}$ is transversal type at $p \in E$, if for some neighourhood $U_p \subset M$ of $p$, $\ov{C_{q}{\cal{F}}} \cap C_{q}{E} = \{0\}$ holds for all $q \in E \cap U_p$. The foliation $\cal{F}$ is transversal type if it is transversal type at each $p \in E$. For a detailed definition, look at \cite{GV2}.

\begin{thm}\label{T:UnifConv}
Let $(W_n,W,U,\cal{F})$ be as before with the additional hypothesis that $W,U$ are bounded taut domains. If $W_n \to W$ in the kernel sense and $\eta_W$ is continuous on $W \sm E$, then $\eta_n \to \eta_W$ uniformly on the compacts of $W \sm (S \cup E)$. If in addition, $\cal{F}$ is transversal type, then $\eta_n \ra \eta_W$ uniformly on the compacts of $W \sm S$.
\end{thm}

\medskip

\no It should be noted that not all points in the defective set $S$ of a foliation $\cal{F}$ are $`$bad'. Given $(W_n, W, U, \cal{F})$, we will say that a point $p \in S$ is a \textit{removable defective point} if $\eta_n(p) \to \eta_W(p)$. We will denote the set of all such points by $\hat{S}$.

\begin{thm}\label{T:RemovableDefectiveSet}
Suppose $W_n, W \subset U$ be bounded connected open sets such that $0 \in W_n, W$ for all $n \in \mbb{N}$, $W$ is taut, and $W_n \to W$ in the kernel sense. Let $\cal{F}$ be a singular hyperbolic Riemann surface foliation on $U$, and $S \subset W$ be its defective set. If $p \in F$ satisfies
\[
\text{int}(L_{p} \cap F) = \emptyset,
\]
where $\text{int}(L_{p} \cap F)$ denotes the interior of the set $L_{p} \cap F$ when considered as a subset of the Riemann surface $L_{p}$, then
\begin{enumerate}
\item $\eta_{n}(q) \to \eta_{W}(q)$ for all $q \in L_{p} \cap W$, that is $L_{p} \cap W \subset \hat{S}$,
\item given additional hypothesis of $U$ being taut, and $\eta_{W}$ being continuous on $W \sm E$, the uniform convergence $\eta_n \to \eta_W$ will be true on the compacts of $(W \sm (S \cup E)) \cup (L_{p} \cap W)$.
\end{enumerate}

\end{thm}

\medskip

\no What is left now is to show that the results proved in \cite{GV1} (Theorem 1.1) and \cite{GV2} (Theorem 2.1) are special cases of Theorem \ref{T:PointConv} and Theorem \ref{T:UnifConv}. To do that lets first recall the type of domain convergence considered in \cite{GV1}, \cite{GV2}. Let $(M,g)$ be a complex hermitian manifold and let $d$ be the distance on $M$ induced by $g$. For 
$S \subset M$ and $\ep > 0$, let $S_{\ep}$ be the $\ep$-thickening of $S$ with distances being measured using $d$. Recall that the Hausdorff distance $\mathcal H(A, B)$ between compact sets $A, B \subset M$ is the infimum of all $\ep >0$ such that $A \subset B_{\ep}$ and $B \subset A_{\ep}$. For bounded domains $U, V \subset M$, the prescription $\rho(U, V) = \mathcal H(\ov U, \ov V) + \mathcal H(\pa U, \pa V)$ defines a metric (see \cite{Bo}) on the collection of all  bounded open subsets of $X$ with the property that if $\rho(U, U_n) \ra 0$, then every compact subset of $U$ is eventually contained in $U_n$, and every neighbourhood of $\ov U$ contains all the $U_n$'s eventually. We will say that $U_n \rightarrow U$ in the Hausdorff sense if $\rho(U_n, U) \to 0$.

\no In the following result $M = \mbb{C}^N$, and $g$ is the Euclidean metric.

\begin{thm}\label{T:Hausdorff-Kernel}
Let $W_n,W$ be bounded connected domains in $\mbb{C}^N$ such that $0 \in W_n, W$ for all $n \in \mbb{N}$. If $W_n \to W$ in the Hausdorff sense, then $W = \text{ker}(W_n)$ and $W_n \to W$ in the kernel sense with $F = \emptyset$.
\end{thm}

\section{Proof of Theorem \ref{T:PointConv}}
Having Proposition 1.5 at our disposal, it suffices to prove that for $p \in W \sm (S \cup E)$ and for some convergent subsequence $\{\eta_{n_k}(p)\}_{k \in \mbb{N}}$, it satisfies
\[
\lim_{k \to +\infty}{\eta_{n_k}(p)} \le \eta_{W}(p).
\]
Fix $p \in W \sm (S \cup E)$. Since $\eta_n(p) \le \eta(p) < +\infty$ (here $\eta$ corresponds to the set $U$), therefore we can choose a subsequence $\{\eta_{n_k}(p)\}_{k \in \mbb{N}}$ which converges to $M$. Consider the corresponding uniformizers $\alpha_{n_k} : \mbb{D} \to L_{p, n_k}$ with $\alpha_{n_k}(0) = p$. Since $\alpha_{n_k}(\mbb{D}) \subset L_{p, n_k} \subset L_{p}$, $\alpha_{n_k}(0) = p$ and $L_p$ being hyperbolic, upto a further subsequence we can suppose that $\alpha_{n_k}$ converges to some holomorphic map $\alpha : \mbb{D} \to L_p$, with $\alpha(0) = p$.

Now we just need to show that $\alpha(\mbb{D}) \subset L_{p, W}$. Suppose there exist a point $q \in \alpha(\mbb{D})$ such that $q \in U \sm \ov{W}$. Let $\alpha(z_0) = q$. Choose a neighourhood $U_{q} \subset U$ of $q$ such that $\ov{U}_{q} \cap \ov{W} = \emptyset$. Take $q_{n_k} := \alpha_{n_k}(z_0)$, since $\alpha_{n_k} \to \alpha$, we have $q_{n_k} \to q$. Therefore there exist $N \in \mbb{N}$ such that $q_{n_k} \in U_{q}$ for all $k \ge N$.
Also $q_{n_k} \in \alpha_{n_k}(\mbb{D}) = L_{p, n_k} \subset W_{n_k}$ for all $k$. Therefore for $k \ge N$
\[
q_{n_k} \in W_{n_k} \sm \ov{W},
\]
and $q_{n_k} \to q$. Thus $q \in F_{\sigma} \subset F$, and since $q \in L_{p}$, therefore $p \in S$, which is a contradiction since $p \in W \sm (S \cup E)$. 

Hence $\alpha(\mbb{D}) \subset \ov{W}$, but $W$ is taut so $\alpha(\mbb{D}) \subset W$. Since $\alpha(\mbb{D}) \subset L_{p}$ is connected, therefore $\alpha(\mbb{D}) \subset L_{p, W}$. Now $\vert \alpha'(0)\vert = \lim_{k \to +\infty}{\vert \alpha_{n_k}'(0)\vert} = \lim_{k \to +\infty}{\eta_{n_k}(p)}$, and by definition of $\eta_{W}$,
\[
\eta_{W}(p) \ge \vert \alpha'(0)\vert.
\]
So we have our result.

\section{Proof of Theorem \ref{T:UnifConv}}

Suppose that there is a compact set $K \subset W \sm (S \cup E)$ and a sequence $\{p_n\} \in K$ such that
\[
\vert \eta_{n}(p_n) - \eta_W(p_n) \vert > \ep
\]
for some $\ep > 0$. 
Assume that $p_n \ra p \in W \sm (S \cup E)$ after passing to a subsequence. 

\medskip

By the continuity of $\eta_W$, $\vert \eta_W(p_n) - \eta_W(p) \vert < \ep/2$
and hence
\begin{equation}\label{E:Eq1}
\vert \eta_{n}(p_n) - \eta_W(p) \vert > \ep/2
\end{equation}
for large $n$. Let $\al_n : \mbb D \ra L_{p_n, W_n} \subset W_n$ be a uniformization map with $\al_{n}(0) = p_n$. Since $W_n,W \subset U$ for all $n$ and $U$ is taut, hence the family $\{\al_n\}$ is normal. By passing to a subsequence, let $\ti \al : \mbb D \ra \ov U$ be a holomorphic limit with $\ti \al(0) = p$. To complete the proof, it suffices to show that $\ti \al : \mbb D \ra L_{p, W}$ is a uniformization. Indeed, if this were the case, then $\eta_W(p) = \vert \ti \al'(0) \vert$ which would then imply that
\[
\eta_{n}(p_n) = \vert \al'_n(0) \vert \ra \vert \ti \al'(0) \vert = \eta_W(p)
\]
and this contradicts (\ref{E:Eq1}). That $\ti{\al}$ uniformizes the leaf $L_{p,W}$ is a consequence of the following observations:

\medskip 

\no (a) $\ti{\alpha}$ satisfies $\ti{\alpha}(\mbb{D}) \subset W$.

\no (b) $\ti{\alpha}$ satisfies $\ti{\alpha}(\mbb{D}) \subset L_{p,W} \cup E$.

\no (c) The map $\ti \al : \mbb D \sm \ti \al^{-1}(E) \ra L_{p,W}$ is a covering map.

\no (d) $\ti{\al}^{-1}(E) = \emptyset$, and $\ti{\al}(\mbb{D}) \subset L_{p,W}$.

\medskip
\no We will only prove the observations $(a)$ and $(b)$, since the proof of other two are almost identical to the similar observations made in \cite{GV1} (see proof of Theorem 1.1).

\medskip

\no (a) Since $U$ is taut and $\ti \al(0) = p \in U$, it follows that $\ti \al (\mbb D) \subset U$. Now suppose that $\ti \al(a) \in U \sm \ov W$ for some $a \in \mbb D$. Then $\al_n(a) \in U \sm \ov W$ for all large $n$, and therefore $\al_n(a) \in W_n \sm \ov W$ which gives us that $\ti \al(a) \in F$. Now since $\ti \al(a) \in L_{p}$, therefore $p \in S$ which is a contradiction to the fact that $p \in W \sm (S \cup E)$. Hence $\ti \al : \mbb D \ra \ov W$. Again, the tautness of $W$ implies that $\ti \al : \mbb D \ra W$.

\medskip

\no (b) To see that $\ti \al(\mbb D) \subset L_{p, W} \cup E$, we will use the following lemma. 

\begin{lem}\label{L:LimitUnif}
Let $\mathcal{F}$ be a hyperbolic SHFC on a complex manifold $M$, with singular set $E \subset M$. If $\{\alpha_n\}_{n \in \mbb{N}}$ is a sequence in $\mathcal{O}(\mbb{D}, \cal{F})$, which converges in the compact parts of $\mbb{D}$ to some $\alpha : \mbb{D} \to M$, then $\alpha(\mbb{D})\subset L \cup E$, where either $L = \emptyset$ or $L$ is a leaf of $\mathcal{F}$. In particular, if $L = \emptyset$, $q \in \alpha(\mbb{D}) \cap E$, and $\mathcal{F}$ is transversal type at $q \in E$, and then $\alpha \equiv q \in E$.
\end{lem}

\begin{proof}
If $\al$ is constant or $\al(\mbb{D}) \subset E$, then we are done. 

Now suppose $\al$ is non-constant, and $p \in \al(\mbb{D}) \cap (M \sm E)$. Consider the open set $N := \{z \in \mbb{D} : \alpha(z) \notin E\} \neq \emptyset$. Define for $w_0 \in N$, the set $N_{w_0} = \{z \in N : L_{z} = L_{w_0}\}$. To see that $N_{w_0}$ is open in $\mbb D$, pick $w_1 \in N_{w_0}$ and a chart $ \phi = (x, y) : \ti W \ra \mbb D \times \mbb D^{n-1}$ around $\ti \al(w_1) = q$ such that $(x, y)(q) = 0 \in \mbb C^n$ and the leaves of $\mathcal F_{\ti W}$ are described by $y = \mbox{constant}$. Let $V$ be the connected component of $\al^{-1}(\ti W) \subset \mbb D$ containing $w_1$. Now, the claim is that $y(\al(V)) = 0$. If this is not the case, then $y(\al (V)) \not= \{0\}$. Since $y(\al(w_1)) = 0$, it follows that $y \circ \al : V \ra \mbb D^{n-1}$ is a non-constant map which implies that its derivative $d(y \circ \al)(w_2) \not= 0$ for some $w_2 \in V$ near $w_1$. Since $\al_n$ converges to $\al$, $d(y \circ \al_n)(w_2) \not= 0$ for $n$ large. This means that $\al_n(\mbb D)$ is not contained in a leaf of $\mathcal F$. This contradiction shows that $y(\al(V)) = 0$. Hence every $w \in V$ has the property that the leaf of $\mathcal F_{\ti W}$ containing $\al(w)$ coincides with $y = 0$ which is also the leaf passing through $w_1$. This shows that $N_{w_0}$ is open.

Now let $F = \alpha^{-1}(E)$. It suffices to show that if $F$ is non-empty, then $F$ is a discrete set. Suppose $F$ is not discrete, and $z \in D$ is a limit point of $F$. Since $F$ is closed, $z \in F$. Consider the point $\alpha(z) = p \in E$. Since $E$ is an analytic subset of $M$, let $f_1, f_2, \ldots, f_m$ be the local defining functions for $E$ near the point $p$, i.e., there exist a neighourhood $V \subset M$ of $p$ such that $V \cap E = \{f_1 = f_2 = \ldots = f_m = 0\}$. Let $r >0$ be such that $\alpha(D(z, r)) \subset V$. Consider the holomorphic functions $f_1 \circ \alpha\vert_{D(z,r)}, f_2 \circ \alpha\vert_{D(z,r)}, \ldots, f_m \circ \alpha\vert_{D(z,r)}$ defined on the disc $D(z,r)$. Since $z$ is an accumulation point of the zero set $\mathcal{Z}(f_i \circ \alpha\vert_{D(z,r)})$ in $D(z,r)$, for all $1 \le i \le m$. Therefore by Identity principle, $f_i \circ \alpha\vert_{D(z,r)} \equiv 0$ on $D(z,r)$ for all $1 \le i \le m$, which in turn gives us that $D(z,r) \subset F$. Now if $F \neq \mbb{D}$, then there exist a point $\tilde{z} \in \partial (\mbb{D} \sm F)$ which is an accumulation point of $F$ (since $F$ is not discrete). But by above argument $\tilde{z}$ belongs to the interior of $F$, which is a contradiction. Thus $F = \mbb{D}$, that is $\alpha(\mbb{D}) \subset E$ which is a contradiction to the assumption that $p \in \al(\mbb{D}) \cap (M \sm E)$. Therefore $F$ is a discrete set.

\medskip
Now if $L = \emptyset$, and consider $q \in \alpha(\mbb{D}) \cap E$ such that $\mathcal{F}$ is transversal type at $q$. We need to show that $\alpha$ is a constant map. Suppose not. Let $z_0 \in \mbb{D}$, $r > 0$ be such that $\alpha(z_0) = q$ and $\alpha(D(z_0,r)) \subset U_{q}$. Since $\alpha\vert_{D(z_0, r)} : D(z_0, r) \rightarrow U_q \cap E$ is non-constant, take $z_1 \in D(z_0,r)$ be such that $\alpha'(z_1) \neq 0$, and since $\alpha_{n}'(z_1) \rightarrow \alpha'(z_1)$, therefore $\alpha'(z_1) \in C_{\alpha(z_0)}\mathcal{F}$. Also since $\alpha(\mbb{D}) \subset E$, therefore $\alpha'(z_1) \in C_{\alpha(z_1)}E$, which is a contradiction to the transversality property of the foliation $\cal{F}$ in the neighourhood $U_{q}$. Therefore, $\alpha \equiv q \in E$. 
\end{proof}

\medskip
\no Since each $\alpha_n \in \cal{O}(\mbb{D}, \cal{F})$, therefore by above lemma $\ti \al : \mbb{D} \to L_{p} \cup E$. From $(a)$, we already saw that $\ti \al(\mbb{D}) \subset W$. Since $p \in \ti \al(\mbb{D})$ and $\mbb{D} \sm (\ti \al)^{-1}(E)$ is connected, therefore $\ti \al : \mbb{D} \to L_{p,W} \cup E$.

\medskip 

Now if $\mathcal{F}$ is of transversal type, in addition to the above hypothesis. To show that the convergence is uniform on compacts of $W \sm S$, it suffices to prove that for $p \in E$, and $p_n \in W \sm (S \cup E)$ such that $p_n \to p$, 
\[
\eta_{n}(p_n) \to \eta_{W}(p) = 0.
\]

\no Suppose not, then $\vert\eta_{n}(p_n)\vert > \ep$. Consider the uniformization map $\alpha_n$ of the leaf $L_{p_n, W_n}$ with $\alpha_n(0) = p_n$. Therefore after passing to a subsequence we get that $\{\alpha_n\}_{n \in \mbb{N}}$ converges uniformly on the compacts of $\mbb{D}$ to a map $\tilde{\alpha} : \mbb{D} \rightarrow W$ with $\tilde{\alpha}(0) = p \in E$. Since $\al_n \in \cal{O}(\mbb{D}, \cal{F})$ for all $n$, therefore by Lemma \ref{L:LimitUnif}, $\ti \al$ is constant and $\ti \al \equiv p$. Thus $\eta_n(p_n) \to 0$ which is a contradiction.

\no Thus, $\eta_{n}$ converges to $\eta_{W}$ uniformly on compact subsets of $W \sm S$.

\section{Proof of Theorem \ref{T:RemovableDefectiveSet}}

\no \textbf{(1)} We need to show $\eta_{n}(q) \to \eta_{W}(q)$ holds true for all $q \in L_{p} \cap W$. To see this, suppose there exist $q \in L_{p} \cap W$ such that $\eta_{n}(q) \not\to \eta_{W}(q)$. Upto a subsequence, we can suppose that $\eta_{n}(q) \to R \neq \eta_{W}(q)$, for some constant $R > 0$. Let $\alpha_{n} : \mbb{D} \to L_{q,n}$ be the uniformizing map such that $\alpha_n(0) = q$. Upto further subsequence, we can suppose that $\alpha_{n_k} \to \alpha : \mbb{D} \to L_{q}$ with $\alpha(0) = q$. Since $W$ is a bounded taut domain, it suffices to show that $\alpha(\mbb{D}) \subset \ov{W}$ to get a contradiction to the assumption $\eta_{n_k}(q) \not\to \eta_{W}(q)$. Now suppose if possible there exist $\ti{q} = \alpha(z_0) \in U \sm \ov{W}$. Take a neighourhood $U_{\ti{q}}$ of $\ti{q}$ such that $\ov{U_{\ti{q}}} \subset U \sm \ov{W}$. Let $r > 0$ be such that $\alpha(D(z_0, r)) \subset U_{\ti{q}}$. Observe that $\alpha(D(z_0, r))$ is connected and since $\alpha$ is not a constant map, therefore $\alpha(D(z_0, r))$ is open in the hyperbolic Riemann surface $L_{p}$. Choose $\epsilon > 0$ such that 
\[
\ov{(U_{\ti{q}})_{\epsilon}} \subset U \sm \ov{W},
\]
where $(U_{\ti{q}})_{\epsilon}$ denotes the $\epsilon-$neighourhood of the domain $U_{\ti{q}}$. Since $\alpha_{n_k} \to \alpha$ uniformly on compacts of $\mbb{D}$, therefore there exist $N_1 > 0$ such that $\alpha_{n_k}(D(z_0,r)) \subset (U_{\ti{q}})_{\epsilon}$ for all $k \ge N_1$. Now clearly $\alpha_{n_k}(z) \in W_{n_k} \sm \ov{W}$ for all $k \ge N_1$ and for all $z \in D(z_0,r)$. This gives us that $\alpha(z) \in F$ for all $z \in D(z_0, r)$, that is $\alpha(D(z_0,r)) \subset F$. So $\alpha(D(z_0,r)) \subset F \cap L_{q} = F \cap L_{p}$, which is a contradiction since $L_{p} \cap F$ has empty interior in $L_{p}$. Thus $\alpha(\mbb{D}) \subset \ov{W}$. 

\medskip

\no \textbf{(2)} Suppose on the contrary that $\eta_{n}$ doesn't converge to $\eta_{W}$ uniformly on the compacts of $(W \sm (S \cup E)) \cup (L_{p} \cap W)$. Let $K \subset W$ be a compact set on which $\eta_n \not\to \eta_W$ uniformly. Then for $\epsilon > 0$ and upto a subsequence, there exist a sequence $q_n \in W$ such that for all $n \in \mbb{N}$
\[
\vert \eta_n(q_n) - \eta_W (q_n)\vert \ge \ep.
\]
Upto another subsequence, suppose $q_n \to q$. Since $\eta_W$ is continuous on $W \sm E$, therefore $\vert \eta_{W}(q_n) - \eta_{W}(q)\vert < \frac{\ep}{4}$, for all $n \ge N$ for some positive $N > 0$. So, we get for all $n \ge N$
\begin{equation}\label{E:Inequality}
\vert \eta_{n}(q_n) -  \eta_{W}(q)\vert \ge \frac{\ep}{2}.
\end{equation}

If $q \in W \sm (S \cup E)$, Theorem \ref{T:UnifConv} gives us a contradiction to the above inequality. So, it suffices to assume that $q \in (L_{p} \cup W)$ and get a contradiction. Let $\alpha_{n} : \mbb{D} \to L_{q_n, n}$ be the uniformization of the leaf $L_{q_n, n}$ in $W_n$, such that $\alpha_n(0) = q_n$. Since $\alpha_{n}(\mbb{D}) \subset U$ for all $n$, and $U$ being taut, therefore $\alpha_n$ converges to a map $\alpha : \mbb{D} \to U$ uniformly on compacts of $\mbb{D}$. Again $\alpha_n \in \cal{O}(\mbb{D}, \cal{F})$ for all $n \in \mbb{N}$, therefore Lemma \ref{L:LimitUnif} tells us that $\alpha(\mbb{D}) \subset L_{q} \cup E$. If we prove that $\alpha(\mbb{D}) \subset L_{q, W} \cup E$, then the exact arguments that we made in the proof of Theorem \ref{T:UnifConv} can be used to show that $\alpha(\mbb{D}) \subset L_{q, W}$. In fact, $\alpha$ will be a uniformization of the leaf $L_{q,W}$, which will give a contradiction to the inequality (\ref{E:Inequality}), thus giving us the result. 

\no Here, it suffices to show that $\alpha(\mbb{D}) \subset \ov{W}$. To see this, note that using tautness of $W$, we have $\alpha(\mbb{D}) \subset W$, and therefore 
\[
\alpha(\mbb{D}) \subset W \cap (L_{q} \cup E).
\]
Now, since $\alpha(\mbb{D})$ is connected and $\alpha(0) = q$, we get $\alpha(\mbb{D}) \subset L_{q, W} \cup E$. 

Suppose if possible that there exist $s = \alpha(z_0) \in U \sm \ov{W}$ for some $z_0 \in \mbb{D}$. As we did in the proof of $(1)$, we choose a neighourhood $U_{s} \subset U$ of $s$ and $\delta > 0$ small enough such that $\ov{(U_{s})_{\delta}} \subset U \sm \ov{W}$. Now choose $r>0$ such that $\alpha(D(z_0, r)) \subset U_{s}$. Since $\alpha_n \to \alpha$ uniformly on compacts of $\mbb{D}$, therefore there exist $N_1 > 0$ such that $\alpha_{n}(D(z_0,r)) \subset (U_{s})_{\delta}$ for all $n \ge N_1$. Observe that for all $z \in D(z_0, r)$ and for all $n \ge N_1$, we have $\alpha_{n}(z) \in W_n \sm \ov{W}$, and therefore $\alpha(z) \in F$. Thus we get 
\[
\alpha(D(z_0, r)) \subset L_{p} \cap F \cap E.
\]

As $\alpha(0) = q \notin U_s$, therefore $\alpha : \mbb{D} \to L_{p} \cup E$ cannot be a constant map. Again, Lemma \ref{L:LimitUnif} tells us that $(\alpha)^{-1}(E) \subset \mbb{D}$ is discrete, and so $\mbb{D} \sm \alpha^{-1}(E)$ is connected. Therefore, the map $\alpha\vert_{\mbb{D} \sm \alpha^{-1}(E)} : \mbb{D} \sm \alpha^{-1}(E) \to L_{p}$ is open. We can choose $0 < \ti{r} < r$ such that $\alpha(D(z_0, \ti{r}) \sm \{z_0\}) \subset L_{p}$ and is open in $L_{p}$, that is $\alpha(D(z_0, \ti{r}) \sm \{z_0\}) \subset L_{p} \cap F$ which is a contradiction to the assumption that $L_{p} \cap F$ has empty interior. Thus $\alpha(\mbb{D}) \subset \ov{W}$ and we have our result.

\section{Proof of Theorem \ref{T:Hausdorff-Kernel}}
We will prove this in 3 steps.
\medskip

\no $(1)$ The first step is to show that $W = \text{ker}(W_n)$. Let $W_0 := \text{ker}(W_n)$ and let $p \in W$. Consider a connected neighourhood $U_p$ of $p$ in $W$ such that $0 \in U_p$ and $\ov{U}_p \subset W$ is compact. Therefore there exist $N_1 > 0$ such that 
\[
\ov{U}_p \subset W_n
\]
for all $n \ge N_1$, which gives us that $\ov{U}_p \subset \cap_{n \ge N_1} W_n = \ti{V}_{N_1}$. Since $0 \in U_p$ and $U_p$ is connected, therefore $p \in V_{N_1} \subset W_0$. Thus $W \subset W_0$.

Now let $p \in W_0$. Therefore, $p \in V_{N_2}$ for some $N_2 > 0$. We can choose a neighourhood $U_p$ of $p$ in $V_{N_2}$ such that $0 \in V_{N_2}$ and $\ov{U}_p \subset V_{N_2}$ is compact. So, we get $\ov{U}_{p} \subset W_n$ for all $n \ge N_2$. Suppose if possible $p \notin \ov{W}$. Therefore there exist $N_3 > 0$ such that $p \notin W_n$ for all $n \ge N_3$ which is a contradiction since $p \in \ov{U}_{p} \subset W_n$ for all $n \ge N_2$. Thus $p \in \ov{W}$. Now suppose if possible $p \in \partial W$. Since $H(\partial W_n, \partial W) \to 0$, therefore for any neighourhood $\hat{U}_p$ of $p$, there exist $N_4 > 0$ such that 
\[
\partial W_n \cap \hat{U}_{p} \neq \emptyset
\]
for all $n \ge N_4$, which is a contradiction since $U_p \cap \partial W_n = \emptyset$ for all $n \ge N_2$. So $W_0 \subset W$.

\no Thus $W = W_0 = \text{ker}(W_n)$.
\medskip

\no $(2)$ In this step, we will show that $W_n \to W$ in the kernel sense. By definition, we need to show that for any subsequence $\{W_{n_j}\}_{j \in \mbb{N}}$, we have $\text{ker}(W_{n_j}) = W$. It is easy to see that $\rho(W_n, W) \to 0$ implies $\rho(W_{n_j}, W) \to 0$ as $j \to +\infty$. Now using Step $(1)$, we get $\text{ker}(W_{n_j}) = W$. Thus $W_n \to W$ in the kernel sense.
\medskip

\no $(3)$ The final step is to show that $F = \emptyset$. Since $F = \cup_{\sigma} F_{\sigma}$. Consider any subsequence $\sigma = \{W_{n_j}\}_{j \in \mbb{N}}$ of the sequence $\{W_n\}_{n \in \mbb{N}}$ such that $F_{\sigma} \neq \emptyset$. Let $p \in F_{\sigma}$ which by definition tells us that $p \notin \ov{W}$. Therefore, there exist a sequence of points $\{p_{n_j}\}_{j \in \mbb{N}}$ such that $p_{n_j} \in W_{n_j} \sm \ov{W}$ for all $j$, and $p_{n_j} \to p$. We can choose a neighourhood $U_p$ of $p$ and $\ep > 0$ small enough so that 
\[
\ov{U}_{p} \cap (W)_{\ep} = \emptyset,
\]
where $(W)_{\ep}$ is the epsilon neighourhood of the domain $W$. Therefore, there exist $N_5 > 0$ such that $\ov{U}_p \cap W_n = \emptyset$, for all $n \ge N_5$. This is a contradiction since $p_{n_j} \in \ov{U}_p \cap W_{n_j}$ for all $n_j \ge N_6$, for some $N_6 > 0$. 

\no Thus $F_{\sigma} = \emptyset$, which gives us that $F = \emptyset$.

\section{Remarks \& Examples}

\no \textbf{Remark 1:} It should be noted that this \textit{defective set} $S$ is not always \textit{very small}. It is possible that $S$ could be dense in $W$. Suppose $U = P(0, (5,5)) \subset \mbb{C}^2$. Consider a countable dense subset $\{q_j\}_{j \in \mbb{N}} \subset \partial W$, where $W = P(0,(1,1))$. Let $L_j$ be the complex line passing through the point $q_j$ and the origin in $U$. Consider a sequence of decreasing neighourhoods $\{R_{j,m}\}_{m \in \mbb{N}}$ of the line $L_j$ for each $j$, such that 
\[
\cap_{m \in \mbb{N}} R_{j,m} = L_j,
\]
for all $j \in \mbb{N}$. Now define $W_{j,m} := P(0,(1,1)) \cup R_{j,m} \subset U$ for all $j,m \in \mbb{N}$. Now we would like to define a sequence of domains $\{W_n\}_{n \in \mbb{N}}$ which contains $W_{j,m}$ for all $j,m \in \mbb{N}$. One way is to start with $W_1 := W_{1,1}$ and follow the process shown in the Figure \ref{fig:Sequence}, with the first few terms would be $W_{1,1}, W_{1,2}, W_{2,1}, W_{3,1}, W_{2,2}, W_{1,3}$, and so on.

\begin{figure}[htp]
    \centering
    \includegraphics[width=10cm]{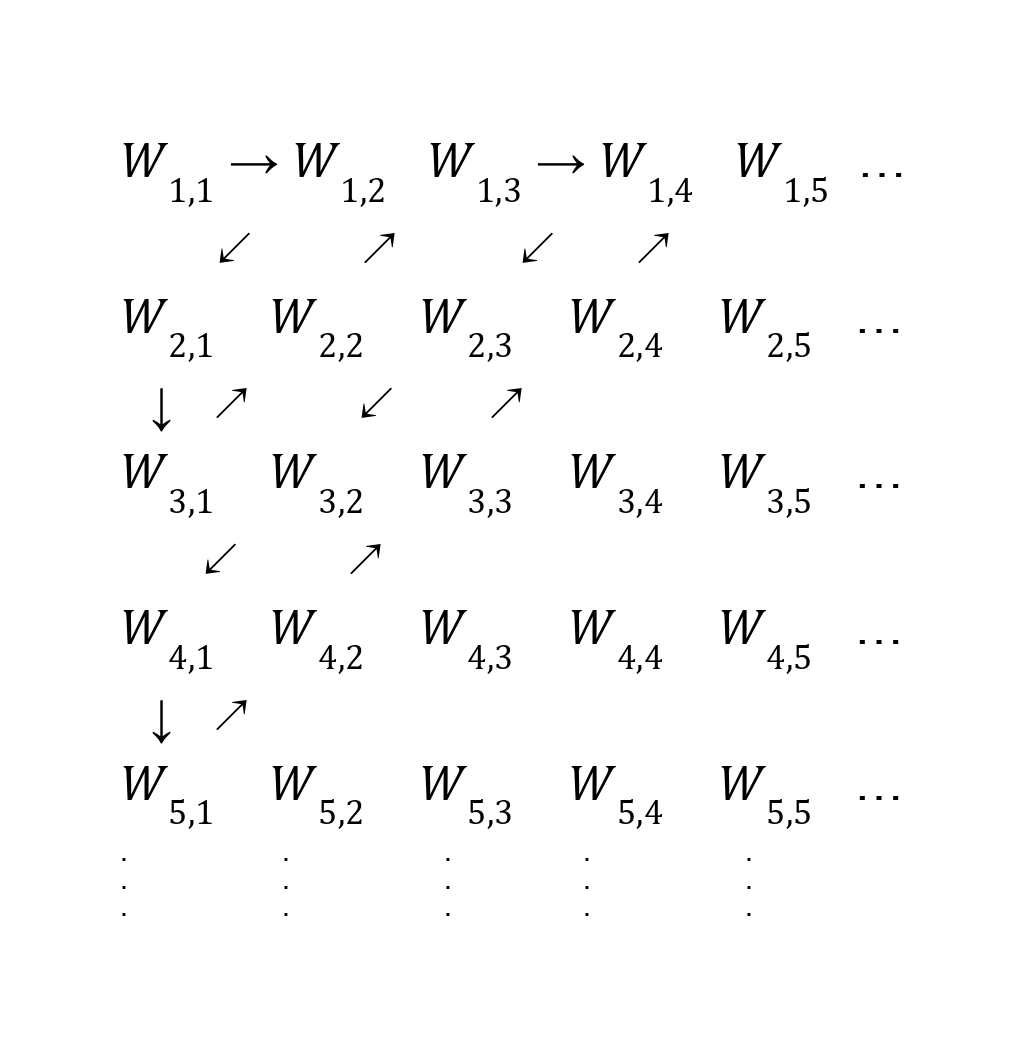}
    \caption{}
    \label{fig:Sequence}
\end{figure}

\no It is easy to see that $W_n \to W$ in the kernel sense, and for each $j \in \mbb{N}$, $\{W_{j,m}\}_{m \in \mbb{N}}$ is a subsequence of $\{W_n\}_{n \in \mbb{N}}$. Therefore, the defective set $S_j = L_j \cap W$ corresponding to the subsequence $\{W_{j,m}\}_{m \in \mbb{N}}$ is subset of the defective set $S$ of the sequence of domains $\{W_n\}_{n \in \mbb{N}}$. Thus $L_j \cap W \subset S$ for all $j \in \mbb{N}$. But note that $\cup_{j \ge 1} (L_j \cap W)$ is a dense subset of W. Therefore, the defective set $S$ is a dense subset of $W$.

\medskip 

\no \textbf{Remark 2:} Although we have proved the results in this note for connected domains in $\mbb{C}^N$, but one can check that everything still follows along the same lines if we suppose $U$ to be an open subset of a complex hermitian manifold $(M,g)$, where $g$ is a smooth hermitian metric on $M$. Also, as we did in \cite{GV1}, we can just work with smooth Riemann surface laminations $\cal{L}$ instead of holomorphic foliations, and all the results still holds true. We proved the results for the case $(M,g) = (\mbb{C}^N, \vert dz\vert^2)$ and a hyperbolic Riemann surface foliation $\cal{F}$ on $U \subset \mbb{C}^N$ just to keep calculations simple.

\medskip

\begin{exam}
Let $\cal{F}$ be the hyperbolic singular holomorphic foliation induced by the vector field 
\[
X = x\frac{\partial}{\partial x} + zy \frac{\partial}{\partial y} + zy \frac{\partial}{\partial z}
\]
on the open polydisc of radius $r = (5,5,5)$ denoted by $U= P(0,r) \subset \mbb{C}^3$. The singular set of $\cal{F}$ is $E = \{(0,y,z) \in U : yz = 0\} = \{y-\text{axis}\}\cup \{z-\text{axis}\}$ (here dim$(E) = 1$). Let $W_n = P(0, (1,1,1)) \cup P(0, (\frac{1}{n}, 2, \frac{1}{n}))$ and $W = P(0, (1,1,1))$. We can check that $W_n \to W$ in the kernel sense with the set $F = \{(0,y,0) \in U : y \in D(0,2) \sm \ov{D(0,1)}\}$. Note that $F \subset E$, therefore the defective set corresponding to the foliation $\cal{F}$ is empty, that is $S = \emptyset$.

As shown in \cite{GV2} (Example 1.17), using the fact that $\Sigma_1 = \{x=0\}, \Sigma_2 = \{y=0\} \; \& \; \Sigma_3 = \{z=0\}$ are $\cal{F}-$invariant to conclude that $\eta_{W}$ is continuous on $W \sm E$. It has also been proved that $\cal{F}$ is transversal type. Therefore, $\eta_{W}$ can be extended continuously on whole of $W$.

\no Since $W,U$ are bounded taut domains in $\mbb{C}^3$, therefore Theorem \ref{T:UnifConv} tells us that $\eta_{n} \to \eta_{W}$ uniformly on the compacts of $W$.

\end{exam}

\medskip

\begin{exam}
Let $\cal{F}$ be the hyperbolic singular holomorphic foliation induced by the vector field 
\[
X = xy\frac{\partial}{\partial x} + zy \frac{\partial}{\partial y} + zx \frac{\partial}{\partial z}
\]
on the open polydisc of radius $r = (5,5,5)$ denoted by $U= P(0,r) \subset \mbb{C}^3$. The singular set of $\cal{F}$ is given by $E = \{(x,y,z) \in U : xy = yz = zx = 0\} = \{x-\text{axis}\} \cup \{y-\text{axis}\} \cup \{z-\text{axis}\}$ (here dim$(E) = 1$). Let $W_n = P(0, (1,1,1)) \cup P(0, (\frac{1}{n}, 2, \frac{1}{n}))$ and $W = P(0, (1,1,1))$. Here $W_n \to W$ in the kernel sense with the set $F = \{(0,y,0) \in U : y \in D(0,2) \sm \ov{D(0,1)}\}$. Note that $F \subset E$, therefore the defective set corresponding to the foliation $\cal{F}$ is empty, that is $S = \emptyset$.

As shown in \cite{GV2} (Example 1.18), using the fact that $\Sigma_1 = \{x=0\}, \Sigma_2 = \{y=0\} \; \& \; \Sigma_3 = \{z=0\}$ are $\cal{F}-$invariant to conclude that $\eta_{W}$ is continuous on $W \sm E$.

\no Take $p= (x,0,0) \in E$ for $x \neq 0$. Observe that $C_{p}E = \langle e_1\rangle$, and therefore $(1,0,0) = e_1 \in C_{p}E$. Consider the sequence $p_n = (x,\frac{1}{n},0) \in W \sm E$ and note that $e_1 \in T_{p_n}\cal{F}$ for all $n \in \mbb{N}$. Therefore $e_1 \in \overline{C_{p}\cal{F}}$. Thus $\cal{F}$ is not transversal type at $p$. Since $p$ is arbitrary, $\cal{F}$ is not transversal type at any point of $\{(x,0,0) : x \in \mbb{D}^{\ast}\} \subset E$.

\no Now on $\Sigma_3 = \{z = 0\}$, the foliation $\cal{F}_{W}$ is given by $X = xy \frac{\partial}{\partial x}$. One can check that the leaf of $\cal{F}_{W}$ passing through $p_n$ is given by $L_{p_n, W} = \{(\xi,\frac{1}{n},0): \xi \in \mbb{D}^{\ast}\} \cong \mbb{D}^{\ast}$ for all $n \in \mbb{N}$. Therefore $\eta_{W}(p_n) = \eta_{W}(p_m) \neq 0$ for all $n,m \in \mbb{N}$. 

\no On $\Sigma_{2} = \{y =0\}$, the foliation $\cal{F}_{W}$ is given by $X= xz\frac{\partial}{\partial z}$. Here the leaf of $\cal{F}_{W}$ passing through $\tilde{p} = (x,0,z)$ is given by $L_{\ti{p},W} = \{(x,0,\xi) : \xi \in D(0,r_3)^{\ast}\} \cong \mbb{D}^{\ast}$. Note that $L_{\ti{p},W} \cup \{p\} \cong \mbb{D}$, i.e., $L_{\ti{p},W}$ is a separatrix through $p$. Take any sequence $\{q_n\}_{n \in \mbb{N}} \subset L_{\ti{p},W}$ such that $q_n \to p$, we have $\eta_{W}(q_n) \to 0$.

\no Therefore, we get two sequences $p_n, q_n \to p$, but $\eta_{W}(p_n) \to k \neq 0$, and $\eta_{W}(q_n) \to 0$. Thus, $\eta_{W}$ cannot be extended continuously to $p$.

We use similar arguments to check that for every $p \in E \sm\{0\} =: \tilde{E}$, $\cal{F}$ is not transversal type at $p$ (in fact $\cal{F}$ is not transversal type for any $p \in E$). Once again, by looking at the structure of leaves in the invariant hyperplanes $\Sigma_{i}$'s, we can conclude that $\eta_{W}$ can not be extended continuously on any point $p \in \tilde{E} = E \sm \{0\}$.

Since $W,U$ are bounded taut domains in $\mbb{C}^3$, therefore Theorem \ref{T:UnifConv} tells us that $\eta_{n} \to \eta_{W}$ uniformly on the compacts of $W \sm E$.

\end{exam}

\medskip

\begin{exam}
Let $\cal{F}$ be the hyperbolic singular holomorphic foliation induced by the vector field 
\[
X = x\frac{\partial}{\partial x} + 2y \frac{\partial}{\partial y}
\]
on the open polydisc of radius $r = (5,5)$ denoted by $U= P(0,r) \subset \mbb{C}^2$. The singular set of the foliation $\cal{F}$ is $E = \{(0,0)\}$. 

Consider the complex line segment $R = \{(x,y) \in P(0,(3,3)) \mid y = x\}$. Let $V_n \subset U$ be a sequence of decreasing open connected neighourhoods of the segment $R$ such that $\cap_{n \ge 1} V_{n} = R$. Now let $W_n = P(0, (1,1)) \cup V_n$, and $W = P(0, (1,1))$. Note that $W_n \to W$ in the kernel sense with the set $F = \{(x,y) \in P(0,(3,3)) \sm \ov{P(0, (1,1))} \mid y = x\}$.

\no Let $p = (x_0, y_0) \in U \sm E$, the leaf $L_{p}$ of $\cal{F}$ is given by the solution of the differential equation 
\[
\frac{dy}{dx} = \frac{2y}{x},
\]
with the initial condition $y(x_0) = y_0$. One can check that 
\[
L_{p} = \{(x,y) \in U \sm E \mid x_{0}^2 y = y_0 x^2\}.
\]
Now if $q = (x,x) \in R \cap L_{p}$, then $x_{0}^2 x = y_0 x^2$. Since $(0,0) \notin L_{p}$, therefore $R \cap L_{p} \neq \emptyset$ if and only if $x_0 y_0 \neq 0$. If $p \in U \sm E$ is such that $x_0 y_0 \neq 0$, then $q = (\frac{x_{0}^2}{y_0}, \frac{x_{0}^2}{y_0})$, that is 
\[
R \cap L_{p} = \{(\frac{x_{0}^2}{y_0}, \frac{x_{0}^2}{y_0})\}.
\]
So, $F \cap L_{p} \neq \emptyset$ if and only if $\vert y_{0}\vert < \vert x_{0}\vert^2$. Consider the set 
\[
T = \{(x,y) \in W \sm E : \vert y\vert < \vert x\vert^2\} \subset W.
\]
Note that if $p \in T$, then we can see from above that $L_{p, W} \subset S$, where $S$ is the defective set in $W$ corresponding to the foliation $\cal{F}$. Since $p \in T$ is arbitrary, therefore $T \subset S$. Now one can easily check from the description of leaves of $\cal{F}$ that for $p \in T$, the leaf $L_{p,W}$ is contained in $T$, that is $L_{p,W} \subset T$. Therefore, we get $S \subset T$. Thus, the defective set $S = T$ is an open set in $W$.

But as we already observed that for every $p \in S$, the intersection $L_{p} \cap F$ is a singleton set. Therefore, for every $p \in S$, 
\[
\text{int}(L_{p} \cap F) = \emptyset,
\]
where $\text{int}(L_{p} \cap F)$ denotes the interior of $L_{p} \cap F$ as a subset of the leaf $L_{p}$. Thus, every $p \in S$ is a removable defective point, that is $\hat{S} = S$.

\no Since $W,U$ are bounded taut domains in $\mbb{C}^2$, \cite{GV2} gives us the continuity of $\eta_{W}$ on $W \sm E$. Furthermore, $\cal{F}$ is transversal type at $0 \in E$, which gives us that $\eta_{W}$ is continuous on whole of $W$. 

\no Theorem \ref{T:RemovableDefectiveSet} now tells us that $\eta_{n} \to \eta_{W}$ uniformly on the compacts of $W$.

\end{exam}

We will end the section with a few questions.

\medskip

\no \textbf{Question 1:} In Theorem \ref{T:RemovableDefectiveSet}, we mentioned a sufficient condition for a defective point $p \in S$ to be a removable defective point, that is $p \in \hat{S}$. But we can easily construct an example of a sequence of domains and foliation to show that the condition in Theorem \ref{T:RemovableDefectiveSet} is not necessary. It would be interesting to know the complete description of the set $\hat{S}$.

\medskip

\no \textbf{Question 2:} Given $(W_n, W, U, \cal{F})$, is it really necessary to assume the continuity of $\eta_{W}$ on whole $W \sm E$ in order to prove the uniform convergence part in Theorem \ref{T:UnifConv} and Theorem \ref{T:RemovableDefectiveSet}? For instance, if the defective set $S \subset W$ is closed subset, then the uniform convergence of $\eta_{n}$ to $\eta_{W}$ still holds on the compacts of $W \sm (S \cup E)$ even if we only assume that the function $\eta_{W}$ is continuous on $W \sm (S \cup E)$. The proof will be exactly the same as we did for Theorem \ref{T:UnifConv} and Theorem \ref{T:RemovableDefectiveSet}.


\begin{thebibliography}{DNS1}

\bibitem{Bo} Boas, Harold P.:
\emph{The Lu Qi-Keng conjecture fails generically}. Proc. Amer. Math. Soc. {\bf 124} (1996), no. 7, 2021--2027.

\bibitem{CLN} Camacho, César; Lins Neto, Alcides:
\emph{Geometric theory of foliations}. Translated from the Portuguese by Sue E. Goodman. Birkhäuser Boston, Inc., Boston, MA, 1985. 

\bibitem{CLS} Camacho, C.; Lins Neto, A.; Sad, P.:
\emph{Minimal sets of foliations on complex projective spaces}. Inst. Hautes Études Sci. Publ. Math. No. {\bf 68} (1988), 187--203 (1989).

\bibitem{Ca} Candel, Alberto:
\emph{Uniformization of surface laminations}, Annales scientifiques de l'\'{E}cole Normale Sup\'{e}rieure {\bf 26} no. 4 (1993), 489--516.

\bibitem{Ch} Chirka, E.M.:
\emph{Complex analytic sets}. Mathematics and its Applications (Soviet Series), Kluwer Academic Publishers Group, Dordrecht, {\bf 46}, 1989 (Translated from the Russian by R. A. M. Hoksbergen).

\bibitem{CG} Candel, A.; Gómez-Mont, X.:
\emph{Uniformization of the leaves of a rational vector field}. (English, French summary)
Ann. Inst. Fourier (Grenoble) {\bf 45} (1995), no. 4, 1123--1133.

\bibitem{DNS1} Dinh, Tien-Cuong; Nguyên, Viet-Anh; Sibony, Nessim:
\emph{Entropy for hyperbolic Riemann surface laminations I}. Frontiers in complex dynamics, 569--592, Princeton Math. Ser., 51, Princeton Univ. Press, Princeton, NJ, 2014.

\bibitem{DNS2} Dinh, Tien-Cuong; Nguyên, Viet-Anh; Sibony, Nessim:
\emph{Entropy for hyperbolic Riemann surface laminations II}. Frontiers in complex dynamics, 593--621, Princeton Math. Ser., 51, Princeton Univ. Press, Princeton, NJ, 2014.

\bibitem{FB} Bacher, François:
\emph{Poincaré metric of holomorphic foliations with non-degenerate singularities}. arXiv:2206.08431

\bibitem{FS} Fornæss, John Erik; Sibony, Nessim:
\emph{Riemann surface laminations with singularities}. J. Geom. Anal. {\bf 18} (2008), no. 2, 400--442.

\bibitem{GV1} Gehlawat, Sahil; Verma, Kaushal:
\emph{Two remarks on the Poincaré metric on a singular Riemann surface foliation}. Complex Var. Elliptic Equ. {\bf 67} (2022), no. 11, 2589--2601.

\bibitem{GV2} Gehlawat, Sahil; Verma, Kaushal:
\emph{Regularity of the leafwise Poincar\'e metric on singular holomorphic foliations} (Submitted). Preprint available at https://arxiv.org/abs/2304.14206

\bibitem{N1} Lins Neto, Alcides:
\emph{Uniformization and the Poincar\'{e} metric on the leaves of a foliation by curves}. Bol. Soc. Brasil. Mat. (N.S.) {\bf 31} (2000), no. 3, 351--366.

\bibitem{N2} Lins Neto, Alcides:
\emph{Simultaneous uniformization for the leaves of projective foliations by curves}. Bol. Soc. Brasil. Mat. (N.S.) {\bf 25} (1994), no. 2, 181--206.

\bibitem{NM} Lins Neto, A.; Canille Martins, J. C.:
\emph{Hermitian metrics inducing the Poincar\'{e} metric in the leaves of a singular holomorphic foliation by curves}. Trans. Amer. Math. Soc. {\bf 356} (2004), no. 7, 2963--2988.

\bibitem{Ng1} Nguyên, Viêt-Anh:
\emph{Singular holomorphic foliations by curves I: integrability of holonomy cocycle in dimension 2}. Invent. Math. {\bf 212} (2018), no. 2, 531--618.

\bibitem{Ng2} Nguyên, Viêt-Anh:
\emph{Ergodic theory for Riemann surface laminations: a survey}. Geometric complex analysis, 291--327, Springer Proc. Math. Stat., {\bf 246}, Springer, Singapore, (2018).

\bibitem{V} Verjovsky, Alberto:
\emph{A uniformization theorem for holomorphic foliations}. The Lefschetz centennial conference, Part III (Mexico City, 1984), 233–253, Contemp. Math., 58, III, Amer. Math. Soc., Providence, RI, 1987.

\end{thebibliography}
\end{document}